\documentclass[lettersize,journal]{IEEEtran}
\usepackage{cite}
\usepackage{amsmath,amssymb,amsfonts}

\usepackage{amsthm}
\usepackage{algorithmic}
\usepackage{graphicx}
\usepackage{textcomp}
\usepackage{flushend}

\usepackage{svg}
\usepackage{array}
\usepackage{url}
\usepackage{verbatim}
\usepackage{multirow}
\def\BibTeX{{\rm B\kern-.05em{\sc i\kern-.025em b}\kern-.08em
    T\kern-.1667em\lower.7ex\hbox{E}\kern-.125emX}}
\usepackage{balance}

\usepackage[english]{babel}
\theoremstyle{definition}
\newtheorem{definition}{Definition}

\newtheorem{assumption}{Assumption}
\newtheorem{rules}{Rule}
\newtheorem{proposition}{Proposition}

\newtheorem{modified_rules}{Rule}
\begin{document}

\title{A Fault-Tolerant Distributed Termination Method for Distributed Optimization Algorithms}

\author{Mohannad Alkhraijah
and Daniel K. Molzahn
 }

\maketitle

\begin{abstract}
This paper proposes a fully distributed termination method for distributed optimization algorithms solved by multiple agents. The proposed method guarantees terminating a distributed optimization algorithm after satisfying the global termination criterion using information from local computations and neighboring agents. The proposed method requires additional iterations after satisfying the global terminating criterion to communicate the termination status. The number of additional iterations is bounded by the diameter of the communication network. This paper also proposes a fault-tolerant extension of this termination method that prevents early termination due to faulty agents or communication errors. We provide a proof of the method's correctness and demonstrate the proposed method by solving the optimal power flow problem for electric power grids using the alternating direction method of multipliers.
\end{abstract}

\begin{IEEEkeywords}
Distributed Optimization, Fault Tolerance, Termination Criteria.
\end{IEEEkeywords}

\section{Introduction}
Distributed optimization algorithms allow multiple computing agents to collaboratively solve large-scale optimization problems by solving smaller subproblems. Distributed optimization algorithms are typically iterative, where agents solve the subproblems in parallel, share their solutions at each iteration, and terminate after finding a solution that satisfies predefined termination criteria. The potential advantages of solving large-scale optimization problems using distributed algorithms include improving scalability by allowing parallel computation, reliability by eliminating a single point of failure, and privacy by reducing the amount of shared data~\cite{SYS-020}.

The advantages of using distributed optimization algorithms motivate their application in electric power systems~\cite{molzahn2017survey}, wireless sensor networks~\cite{10.1145/984622.984626}, robotics~\cite{halsted2021survey}, machine learning~\cite{9446488}, etc. Realizing all potential benefits of distributed optimization in many of these applications requires a fully distributed termination method that eliminates the need for a central coordinator to check the termination statuses of all agents. However, to the best of our knowledge, the literature lacks a fully distributed method for terminating general distributed optimization algorithms.

A fully distributed optimization algorithm implementation requires the use of a distributed termination method. By \emph{distributed termination}, we mean a method that allows agents to agree that a predefined global termination criterion is satisfied and the algorithm can thus terminate despite only having their local information and data communicated with immediate neighbors. Finding an appropriate termination method is a challenging problem in distributed computing in general because no agent has direct access to the global system's convergence status~\cite{fokkink2018distributed}. In addition, it is desirable that termination methods for distributed optimization algorithms are symmetrical with common rules across all agents, require minimal prior knowledge of the system, and have no restrictions on the communication network topology.

\subsection{Related Work}
Terminating computations has been extensively studied in the context of distributed computing. The first methods for terminating distributed computations run in a combination of series and parallel sequences were independently proposed in~\cite{francez1980distributed} and~\cite{dijkstra1980termination}. In the distributed termination problem of distributed computing, there are multiple processors, each with a predecessor and successors. A processor is active when it receives a message from a predecessor. Once the processor completes its computations, the processor becomes an idle processor until a new message is received. A termination method stops the computation when all the processors are idle and no messages are in transit~\cite{fokkink2018distributed}.

Several distributed termination methods have been proposed in the literature to terminate distributed computing processors~\cite{MATOCHA1998207}. Reference~\cite{dijkstra1980termination} proposes a method that uses a rooted tree to track the global termination status. Another widely used approach proposed in~\cite{mattern1989global} uses credit distribution, where each processor receives a credit from its predecessors when activated, and returns the credit when the computations are completed. Another family of distributed termination methods uses a wave-like message passing of the termination status~\cite{francez1980distributed,rana1983distributed}. Reference~\cite{rana1983distributed} proposes using a wave-based method that passes clock-like information and requires defining a Hamiltonian cycle that traverses all processors. A similar method proposed in~\cite{feijen1999shmuel} uses a predefined ring network to pass the termination status across the processors.

A method proposed in~\cite{shavit1986new} combines ideas from~\cite{dijkstra1980termination} and~\cite{rana1983distributed} by using a forest, i.e., multiple rooted trees, in addition to a wave circulating between the trees' roots. A method proposed in~\cite{4339197} uses a combination of credit distribution and a rooted tree, where the root processor maintains a ledger and counts the credits of the successors and predecessors in each level of the tree. Reference~\cite{mattern1987experience} proposes a ``vector method'' where a vector containing information about the termination status of each processor is communicated through a ring network until the termination criteria are satisfied.

An important property for distributed termination methods is \emph{fault tolerance}, i.e., the ability of an algorithm to operate reliably even with the existence of a faulty processor~\cite{MATOCHA1998207}. A fault-tolerant method proposed in~\cite{misra1983detecting} uses multiple message passing at the same time to prevent faulty termination. The authors of~\cite{lai1995n} propose using a variation of the rooted tree method by having multiple roots of the rooted tree and a mechanism for reorganizing the tree's nodes. Most recently, \cite{karlos2021fault} proposes a modification of the ring method in~\cite{feijen1999shmuel} that regularly checks the processors' statuses and reorganizes the ring's network upon detecting a faulty processor.

Although these distributed termination methods are widely used in distributed computing, they are not suitable for distributed optimization algorithms. The main difference between distributed optimization and distributed computing is that the agents in distributed optimization algorithms run simultaneously, i.e., there are no successors and predecessors, and no agent will be idle until the algorithm terminates. Further, fully distributed optimization algorithms may not have a central agent that initiates and terminates the algorithm similar to~\cite{dijkstra1980termination,mattern1989global, 4339197}, nor a communication network topology restriction similar to~\cite{khokhar2003termination, rana1983distributed, shavit1986new, baker2006efficient, karlos2021fault, feijen1999shmuel, mattern1987experience}.

The only approach we are aware of that roughly discusses a similar problem is proposed in~\cite{9186330} to solve distributed model predictive control problems. The method in~\cite{9186330} uses a gossip-based method~\cite{5513273} to calculate the global termination criteria. Although the method is more suitable for distributed control than optimization, the same concept could be extended to distributed optimization algorithms. However, extending the method requires a problem-specific formulation to design the local termination criteria. Moreover, the method in~\cite{9186330} solves the gossip-based algorithm after each iteration, thus adding a significant computational and communications burden, and susceptible to faults that may cause the algorithm to terminate before satisfying the termination criterion. 

\subsection{Contributions}
This paper proposes a fully distributed termination method for a general distributed optimization algorithm. The proposed method is fully distributed and guaranteed to terminate the algorithm after the global termination criterion is satisfied. The method uses simple rules that are symmetrical across all agents, requires minimal prior knowledge of the system, and does not impose any restrictions on the communication network topology. Furthermore, we propose an extension to the method that is fault-tolerant to early termination. The paper presents a proof of the method's correctness as well as numerical verification.

To summarize, this paper: 
\begin{itemize}
\item Motivates the importance of a fully distributed termination method for the practical implementation of distributed optimization algorithms,
\item Proposes a termination method for fully distributed optimization algorithms,
\item Provides a fault-tolerant extension that prevents early termination caused by a fault,
\item Proves the correctness of the proposed method and the fault-tolerant extension, and
\item Demonstrates the effectiveness of the proposed method using practical examples of optimal power flow problems for distributed optimization of electric power grids.
\end{itemize}

\subsection{Organization}
The rest of the paper is organized as follows: Section~\ref{sec2:problem_formulation} presents the mathematical notations and defines the distributed optimization termination problem. Section~\ref{sec3:methodology} presents the proposed method and the correctness proof. Section~\ref{sec4:fault_termination} describes a variant of the proposed method that prevents faulty termination. Section~\ref{sec5:result} shows test cases and numerical verification. Section~\ref{sec6:discussion} provides remarks and extensions, and Section~\ref{sec7:conclusion} presents conclusions.

\section{Distributed Termination Problem} \label{sec2:problem_formulation}

This section introduces notation and states the distributed termination problem with the help of several definitions.

\subsection{Notation}
We represent the communication network with a graph $G(\mathcal{A},\mathcal{E})$, where $\mathcal{A}$ is the set of agents and $\mathcal{E}$ is the set of communication links. We use $\mathcal{N}(i)\subset \mathcal{A}$ to denote $\left\lbrace j \mid (i,j)\in \mathcal{E}\right\rbrace$, the set of neighboring agents that are connected to agent~$i$ by a communication link (excluding agent~$i$ itself). We define a path between agent~$i$ and $j$ as $ P(i,j) := \left\lbrace k_0, k_1, k_2, \ldots, k_{N} \right\rbrace$, a sequence of distinct agents such that $k_0 = i, k_{N} =j$, $(k_n,k_{n+1}) \in \mathcal{E}$, and the path length $N$, where the path length is the number of communication links in the path. We define the distance $d(i,j)$ as the length of the shortest path between agents~$i$ and $j$ and the diameter of a network as $D := \displaystyle \max_{i,j \in \mathcal{A}} \left\lbrace d(i,j) \right\rbrace$, the maximum distance among all the pairs of agents. We denote the iteration index with $t\in \mathcal{T}$, where $\mathcal{T}=\left\lbrace 0, 1, 2, 3, \ldots, \overline{T}_G \right\rbrace$ is the set of iterations before the algorithm terminates, and $\overline{T}_G$ is the maximum number of iterations. We use $|\,\cdot\,|$ to denote the cardinality of a set, and $\bigcup$ to denote the union of two sets. We use $\mathcal{A}\setminus \mathcal{A'}$ to indicate the set of agents in~$\mathcal{A}$ and not in~$\mathcal{A'}$. For any general variable $X$, the superscript $t$ in $X^t$ denotes variables that are known at iteration $t$, the subscript $i$ in $X_i$ denotes variables belonging to agent~$i$, and subscript $G$ in $X_G$ denotes global variables that no agent can directly access. If $X$ is a vector, then $X[i]$ is a scalar that denotes the $i$-th entry of the vector~$X$.

\subsection{Problem Statement and Definitions}
Consider multiple agents solving an optimization problem using an iterative distributed optimization algorithm. The distributed agents are connected via a communication network represented by an undirected graph $G(\mathcal{A}, \mathcal{E})$, where $\mathcal{A}$ and $\mathcal{E}$ are the sets of agents and communication links. At each iteration, the agents perform local computations and share information about their local status via the communication network. Agent~$i$ communicates with agent~$j$ if and only if $(i,j)\in \mathcal{E}$. At the end of each iteration, the agents decide whether to terminate the algorithm or proceed to the next iteration based on local rules. The distributed optimization algorithm should only terminate after the agents satisfy a global termination criterion. We have the following definitions: 

\begin{definition}[D1]
\emph{Local termination criterion}: A condition defining when an agent reaches consensus on the computation results with the neighboring agents. We denote the status of the local termination criterion of agent $i$ at iteration $t$ by $B^t_i \in \left\lbrace 0, 1 \right\rbrace$, where $B^t_i=1$ when agent $i$ satisfies the local termination criterion at iteration $t$, and $B^t_i=0$ otherwise. 
\end{definition}

The predefined condition of the local termination criterion involves reaching a consensus with neighboring agents on some shared quantities within a predefined tolerance. These quantities are typically shared variables' values or the amount by which coupling constraints between neighboring agents are violated. Shared quantities may also include dual residuals that measure the optimality of the solutions. Agents measure consensus using a norm of the mismatch or violations with the neighboring agents. See, e.g.,~\cite[Section~3.3]{boyd2011distributed} for an example of a local termination criterion. Each agent evaluates their local termination criterion using available information from local computations and data communicated with immediate neighbors. 

\begin{definition}[D2]
\emph{Global termination criterion}: A condition defining when all of the agents reach consensus on the distributed optimization results. The global termination criterion is satisfied when all agents satisfy their local termination criteria. We denote the status of the global termination criterion at iteration $t$ by $B^t_G \in \left\lbrace 0, 1 \right\rbrace$, where $B^t_G =1$ when all agents satisfy their local termination criteria at iteration $t$ (i.e., $B^t_i =1$,  for all $i \in \mathcal{A}$), and $B^t_G=0$ otherwise.
\end{definition}

Satisfying the global termination criterion indicates that the distributed optimization algorithm has reached the desired results, and agents should terminate the computations. Since the global termination criterion depends on information from every agent, its status is not directly available to any of the agents. Thus, agents should communicate the global termination criterion and agree on terminating the algorithms.

\begin{definition}[D3]
\emph{Appropriate algorithm termination}: The agents simultaneously terminate their computations after the global termination criterion is satisfied.
\end{definition}

In other words, we say that an algorithm terminates \emph{appropriately} if the following two conditions are satisfied upon termination: 1)~the global termination criterion is satisfied (i.e., all agents' local termination criteria are satisfied) and 2)~all agents terminate their computations simultaneously. More formally, if the global termination criterion is satisfied at iteration $T_G$, then an appropriate algorithm termination implies that $T_G \leq \overline{T}_i = \overline{T}_G$, for all $i \in \mathcal{A}$, where $\overline{T}_i$ is the final iteration of agent~$i$ and $\overline{T}_G$ is the last iteration of any agent. Thus, the algorithm may still terminate appropriately if the agents continue their computations for additional iterations after the global termination criterion is satisfied. While the additional $\overline{T}_G - T_G$ iterations are not necessary for the convergence of the distributed algorithm, they may be needed to ensure an appropriate algorithm termination.

\section{Distributed Termination Method} \label{sec3:methodology}
This section presents our proposed distributed termination method. This method is based on a set of rules that agents apply at each iteration to either certify that the global termination criterion has been met and they can thus terminate or that they should proceed to the next iteration. The agents apply these rules after the communication step of the distributed algorithm, prior to starting a new iteration. These rules depend strictly on information available to each agent via local computations and communication with neighboring agents. The proposed method combines ideas from~\cite{rana1983distributed}, in the sense of passing the termination status using a wave-like message with a clock-like stamp, and the vector method from~\cite{mattern1987experience} since the agents' termination statuses are stored in a vector.

Our method relies on the following three assumptions:
\begin{assumption}[A1]
The communication network graph is connected, i.e., for each pair of agents~$i$ and~$j\in\mathcal{A}$, there is a path $P(i,j)$.
\end{assumption}

\noindent This assumption implies that information communicated by an agent can traverse all other agents.

\begin{assumption}[A2]
All agents know the total number of agents $|\mathcal{A}|$ and the diameter of the communication network~$D$. No other prior knowledge is required.
\end{assumption}

\noindent This assumption implies that agents have a minimal prior knowledge of the communication network topology and the total number of other agents. The only additional knowledge each agent has comes from their local computations and the information communicated by neighboring agents. 

\begin{assumption}[A3]
The local termination criterion is defined by a monotonic process, i.e., $B_i^{t} \leq B_i^{t^{\prime}}$ if $t \leq t^{\prime}$.
\end{assumption}

\noindent In other words, after being satisfied, an agent's local termination criterion remains satisfied for all subsequent iterations. 

\subsection{Method Procedure}
We next present our proposed distributed termination method. Each agent~$i$ stores a local vector $V_i\in \left\lbrace 0,1 \right\rbrace^{|\mathcal{A}|}$, called a \emph{termination vector}, and a local scalar $T_i\in \mathcal{T}$, called a \emph{termination iteration scalar}. We call the combination of both $V_i$ and $T_i$ the \emph{termination status} for agent~$i$. The termination vector $V_i$ stores agent~$i$'s information about the satisfaction of every agents' local termination criterion. The termination iteration scalar $T_i$ stores agent~$i$'s information regarding the iteration number when the last agent satisfies its local termination criterion. Let $t$ be the current iteration and $B^t_i$ be the status of the local termination criterion of agent~$i$ at iteration $t$ as defined in Definition~D1. At $t = 0$, agent~$i$ uses the following rule to initialize its termination status:
\begin{rules}[R0]
Set $T^0_i = 0$ and $V^0_i[k]= 0$ for all $k \in \mathcal{A}$.
\end{rules}
At iteration~$t \geq 1$, agent~$i$ updates its local termination criterion's status $B_i^t$ and receives the termination statuses $V^{t-1}_j$ and $T^{t-1}_j$ from neighboring agents~$j \in \mathcal{N}(i)$. Then, agent~$i$ applies the following rules:
\begin{rules}[R1]
Set $T^{t}_i = \displaystyle \max_{j\in \mathcal{N}(i)} \left\lbrace T^{t-1}_j \right\rbrace$ and $V^{t}_i[k] = \displaystyle \max_{j\in \mathcal{N}(i)}\left\lbrace V^{t-1}_j[k] \right\rbrace$, for all $k \in \mathcal{A} \setminus \left\lbrace i \right\rbrace$.
\end{rules}
\begin{rules}[R2]
If $V^{t-1}_i[i] = 0$ and $B^{t}_i = 1$, then set $T^{t}_i = t$ and $V^{t}_i[i] = 1$.
\end{rules}
\begin{rules}[R3]
If $\sum_{k\in \mathcal{A}}{V^{t}_i[k]} = |\mathcal{A}|$ and $t \geq T^{t}_i + D$, terminate. 
\end{rules}%
\noindent Recall that $D$ in Rule~R3 is the diameter of the communication network, which we assume each agent knows due to Assumption~A2. Rule~R1 uses communications from neighboring agents to update an agent's local termination status. Rule~R2 uses an agent's local information to update its local termination status when the local termination criterion is first satisfied. Rule~R3 dictates when the agents terminate their computations. The four rules R0--R3 ensure that the distributed optimization algorithm will terminate appropriately after $T_G + D$ iterations. In other words, the agents will simultaneously terminate computations at $D$ iterations (diameter of the communication network) after satisfying the global termination criterion (which occurs at iteration $T_G$). Using this method, the agents only need to store a binary vector of size $|\mathcal{A}|$ ($V_i\in \{0,1\}^{|\mathcal{A}|}$) and a single non-negative integer ($T_i\in \mathbb{Z}_{\ge0}$).

\subsection{Method Correctness}
The algorithm is \emph{correct} if the four rules R0--R3 ensure that the algorithm will appropriately terminate; that is, agents following R0--R3 terminate simultaneously after satisfying the global termination criterion. To prove the correctness of the algorithm, we use the following three propositions.

\begin{proposition}[P1]
For any agents~$i$ and $j \in \mathcal{A}$, the $j$-th entry of the termination vector of agent~$i$ is one (i.e., $V^{t}_i[j] = 1$) if and only if agent~$j$ satisfies its local termination criterion (i.e., $B^{t}_j= 1$). 
\end{proposition}

\begin{proof}
For the forward implication, since the value of the termination vector is initialized as $V^0_i[j] = 0$, for all $i,j\in \mathcal{A}$ due to Rule~R0 and $V^{t}_i[j] = 1$, then the value of $V^{t}_i[j]$ changes at or before iteration $t$. Without loss of generality, assume $V^{t_i-1}_{i_0}[j] = 0$ and $V^{t_i}_{i_0}[j] = 1$ at iteration $t_i \leq t$ for some agent~$i_0 \in \mathcal{A}$. This change occurs due to either Rule~R1 or Rule~R2. If this change occurs due to Rule~R2, then $i_0 = j$ since Rule~R2 only changes the same agent's status, and we can thus directly conclude from Rule~R2 that agent~$j$ satisfies its local termination criterion at iteration $t_i$ (i.e., $B^{t_i}_j = 1$). Since $t_i \leq t$, then $B^{t}_j = 1$ due to Assumption~A3, which completes the forward implication for this case. Alternatively, if the change occurs due to Rule~R1, then $i_0 \neq j$ since Rule~R1 only changes the status associated with another agent. This implies $V^{t_i-2}_{i_1}[j] = 0$ and $V^{t_i-1}_{i_1}[j] = 1$, for some agent~$i_1\in \mathcal{N}(i_0)$. Using a similar argument to the one we used for agent~$i_0$, we conclude either agent~$i_1 = j$, which implies $B^{t}_j = 1$, or there is $V^{t_i-3}_{i_2}[j] = 0, V^{t_i-2}_{i_2}[j] = 1$, for some agent~$i_2\in \mathcal{N}(i_1)$ and~$i_2\notin\mathcal{N}(i_0) \bigcup \left\lbrace i_0 \right\rbrace$. We recursively use the same argument to obtain a sequence $\mathcal{I} = \left\lbrace i_0, i_1, i_2, \ldots \right\rbrace \subseteq \mathcal{A}$. The sequence $\mathcal{I}$ is finite since $\mathcal{A}$ is a finite set and each agent can appear no more than once. Consider the last agent in $\mathcal{I}$, denoted as agent~$i'$, with $V^{t_{i'}}_{i'}[j] = 1$ and $t_{i'} = t_i-(|\mathcal{I}|-1) \leq t$. Agent~$i'$ must have used Rule~R2 to update $V^{t_{i'}}_{i'}[j]$, since it is the last agent in $\mathcal{I}$; otherwise, if $i'$ used Rule~R1, then $i'$ would not be the last agent in $\mathcal{I}$. Since $i'$ used Rule~R2, then $i' = j$ and $B^{t_{i'}}_j = 1$. Thus, $B^{t}_j = 1$ due to Assumption~A3 and $t_{i'} \leq t$, which completes the proof of the forward implication. The backward implication is a direct consequence of Rule~R2. If agent~$j$ satisfies the local termination criterion $B^t_j = 1$, then $V^t_j[j] = 1$ due to Rule~R2.
\end{proof}

\begin{proposition}[P2]
If, for any agent~$i\in\mathcal{A}$, all terms in the termination vector are equal to one (i.e., $\sum_{k\in\mathcal{A}} V_i^t[k] = |\mathcal{A}|$), then the distributed optimization algorithm satisfies the global termination criterion (i.e., $B_G^t = 1$).
\end{proposition}

\begin{proof}
Using Proposition~P1, if $\sum_{k\in \mathcal{A}}{V^t_i[k]} = |\mathcal{A}|$, then all agents satisfy their local termination criterion. Thus, the agents satisfy the global termination criterion by Definition~D2.
\end{proof}

\noindent Note that the reverse is not necessarily true since we might have a case where all the agents satisfy their local termination criteria, but some agents have not received the updated local termination status yet.

\begin{proposition}[P3]
If the agents satisfy the global termination criterion at iteration~$T_G$, then the algorithm terminates appropriately in $T_G + D$ iterations.
\end{proposition}

\begin{proof}
Without loss of generality, consider the case when all the agents satisfy their local termination criteria at iteration~$t^{\prime}$ except for agent~$i$, i.e., $B^{t^{\prime}}_j=1$, for all $j\in \mathcal{A} \setminus \left\lbrace i \right\rbrace$. If agent~$i$ satisfies the local termination criterion at iteration $T_G$, then $T^{t}_i = T_G$, for all $t \geq T_G$ due to Rule~R2 and Assumption~A3. By communicating $T_G$ and using Rule~R1 to update the iteration termination scalar in the next iterations by all agents, we get $T^t_j = T_G$, for all $j\in \mathcal{A}, t_j \geq T_G + D$, since the number of iterations needed to communicate a status update between agents~$i$ and~$j$ is equal to the distance between the agents $d(i,j)$ and the distance between any two agents is less than or equal to the diameter of the communication network, i.e., $d(i,j) \leq D$. The analogous argument applies to updates for $V^t_j[k],$ for all $ j,k\in \mathcal{A}$. This implies $T^t_j = T_G$ and $V^t_j[k] = 1$, for all $j,k\in \mathcal{A}, t\geq T_G + D$. Thus, all agents receive the updated global termination criterion status using Proposition~P2 and terminate appropriately at iteration $T_G + D$ using Rule~R3.
\end{proof}

\noindent The last proposition shows that the proposed method terminates the distributed optimization appropriately and thus the method is correct.

\section{Fault-Tolerant Distributed\\ Termination Method} \label{sec4:fault_termination}
The termination method described in Section~\ref{sec3:methodology} assumes all the agents are reliable and share accurate termination statuses. However, agents may fail to properly update termination statuses due to communication errors or other faults. Not following the update rules may either cause early termination or failure to terminate. This section extends the method in Section~\ref{sec3:methodology} to prevents early termination. Formally, we define:

\begin{definition}[D4]
\emph{Faulty termination} (also called \emph{early termination)}: One or more agents terminate their computations before the global termination criterion is satisfied.
\end{definition}

Faulty termination could occur if there were faulty agents not following Rule~R1 correctly by assigning a false positive or false negative termination status to one or more entries in their termination vector. Thus, Proposition~P1 would not be true. We use the following definition: 

\begin{definition}[D5]
\emph{Fault injection}: Any update to the termination status that does not follow the termination method's rules.
\end{definition}

For notational convenience, we call an agent that injects a faulty status a \emph{faulty} agent and an agent whose termination status is incorrectly reported a \emph{faulted} agent. That is, if agent $i$ falsely reports that agent $j$ has reached its termination criterion in violation of Rule~R1, we say that agent $i$ is faulty and agent $j$ is faulted.

Before stating additional rules, we identify two sufficient conditions that indicate the existence of a faulty termination status. We then use these two conditions to correct any faulty termination status in the fault-tolerant method.  

\begin{proposition}[P4]
If, for some agents~$i$ and $k \in \mathcal{A}$ and iteration $t$, the $k$-th entry of the termination vector of agent~$i$ is $V^{t-1}_i[k] = 1$ and the termination vector for some neighboring agent~$j\in \mathcal{N}(i)$ is $V^{t}_j[k] = 0$, then there is a faulty termination status.
\end{proposition}

\begin{proof}
Assume Proposition~P4 is not true, that is, there is no faulty termination status, $V^{t-1}_i[k] = 1$, and $V^{t}_j[k] = 0$, $j\in \mathcal{N}(i)$ for some $i$ and $k \in \mathcal{A}$. If $V^{t-1}_i[k] = 1$, then $V^{t}_j[k] = 1$, for all $j\in \mathcal{N}(i)$ due to Rule~R1, which contradicts our assumption that there is $j\in \mathcal{N}(i)$ with $V^{t}_j[k] = 0$. Thus, there is a faulty termination status. 
\end{proof}

\begin{proposition}[P5]
If, for some agents~$i$ and $k \in \mathcal{A}$ and iteration $t$, the termination vectors are $V^{t-1}_i[k] = 0$ and $V^{t}_i[k] = 1$, and the termination iteration scalar is either $T^t > t$ or $T^t < t-D$, then there is a faulty termination status. 
\end{proposition}

\begin{proof}
The proof of the first condition $T^t > t$ is trivial since it is always true that $T^t \leq t$ when agent~$i$ updates any entry of the termination vector. For the second condition, we use a similar proof to Proposition~P3 to reach a contradiction. Assume $V^{t-1}_i[k] = 0$, $V^{t}_i[k] = 1$, and $T^t < t - D$. This implies that agent~$k$ satisfies the local termination criterion before iteration~$t - D$. Assume agent~$k$ satisfies the local termination criterion at iteration~$t_k$, i.e., $V^{t_k-1}_{k}[k] = 0$ and $B^{t_k}_k = 1$. Thus, $V^{t_k}_{k}[k] = 1$ and $T^{t_k}_k = t_k$ due to Rule~R2. Similar to the proof of Proposition~P3, we can show that $V^{t_k+d(k,i)-1}_i[k] = 0$, $V^{t_k+d(k,i)}_i[k] = 1$ and $T^{t_k+d(k,i)}_i \geq t_k$, for all $i \in \mathcal{A}$, where the inequality is due to the possibility that other agents may satisfy their local termination criteria after iteration~$t_k$. This implies that, at iteration $t = t_k + d(k,i)$, the local termination statuses of agent~$i$ are $V^{t-1}_i[k] = 0$, $V^{t}_i[k] = 1$, and $T^t \geq t - d(k,i)$. Since $d(i,k) \leq D$, then $T^t \geq t - D$, which contradicts the assumption that $T^t < t - D$. Thus, there is a faulty termination status.  
\end{proof}

Propositions~P4 and P5 provide sufficient conditions to diagnose the correctness of the termination statuses. Proposition~P4 detects discrepancies in the updates of the termination vectors, and Proposition~P5 bounds the updates of the termination iteration scalar. We build on Propositions P4 and P5 to develop a fault-tolerant termination method that self-correct any faulty termination status. For the fault-tolerant method to be effective, we need two additional assumptions.

\begin{assumption}[A4]
The faulty agents do not form a cut-set of the communication network.
\end{assumption}

A graph cut-set is any set $\mathcal{C} \subset \mathcal{A}$ such that removing the agents in $\mathcal{C}$ from the graph $G(\mathcal{A}, \mathcal{E})$ results in a disconnected subgraph. Assumption~A4 ensures the existence of a path between the faulted agent and other agents without passing through the faulty agents. This path permits the correct termination statuses to traverse all agents. When the faulty agents form a cut-set, they divide the communication network into two parts, and there is no way for one part to know the correct termination statuses of the other part without passing the termination status through the faulty agents.

\begin{assumption}[A5]
The distributed optimization computations are correct and accurate.
\end{assumption}

Although faults can occur in both the distributed optimization and the termination method, the proposed method focuses on faults in the termination method. We refer the reader to our other work~\cite{10012244, harris_molzahn-hicss2024, alkhraijah2023detecting} and the work in~\cite{munsing2018cybersecurity, 8629934, 9539887, 9762779, 10374181
} on detecting data manipulation in distributed optimization algorithms. Assumption A5 implies that each agent shares the accurate status for its own local termination, but not necessarily the accurate termination statuses of other agents. In Section~\ref{sec:revising_assumptions}, we will show that this assumption can be relaxed to some extent while still ensuring that the fault-tolerant method will terminate the distributed algorithm appropriately.

\subsection{Fault-Tolerant Method Procedure}
The fault-tolerant extension incorporates the conditions in Propositions P4 and P5 to clear any faulty termination status. We modify the rules in Section~\ref{sec3:methodology} and also add three new rules to prevent faulty termination. In additional to the termination vector $V_i\in\left\lbrace 0, 1 \right\rbrace^{|\mathcal{A}|}$ and termination iteration scalar $T_i\in\mathcal{T}$, agents store a \emph{termination iteration vector} $U_i\in \mathcal{T}^{|\mathcal{A}|}$ and a \emph{correction vector} $C_i \in \mathbb{Z}_{\ge0}^{|\mathcal{A}|}$. The termination iteration vector stores the iteration number when each agent satisfies its local termination criterion, and the correction vector tracks the number of iterations that are needed to reject any new update after a faulty status is detected.

Let $t$ be the current iteration and $B_i^t$ be the status of agent~$i$'s local termination criterion at iteration $t$. At iteration~$t=0$, agent~$i$ initializes the local termination status as follows:
\begin{modified_rules}[R0']
Set $V^0_i[k]= 0$, $U^0_i[k]=0$, and $C^0_i[k] = 0$, for all $k \in \mathcal{A}$, and $T^0_i = 0$.
\end{modified_rules}
At iteration $t \geq 1$, agent~$i$ receives the termination statuses $V_j^{t-1}$, $U_j^{t-1}$, and $T_j^{t-1}$ from neighboring agents~$j\in\mathcal{N}(i)$ and applies the following rules:

\begin{modified_rules}[R1']
If $B^{t}_i = 1$, $V^{t-1}_i[i] = 0$, and $C^{t-1}_i[i] =0$, then set $V^{t}_i[i] = 1$, and $U^{t}_i[i] = t$.
\end{modified_rules}
\begin{modified_rules}[R2']
For each $k \in \mathcal{A} \setminus \left\lbrace i \right\rbrace$, let $\mathcal{J}_k = \left\lbrace j \mid j \in \mathcal{N}(i), V^{t-1}_j[k] = 1, U^{t-1}_j[k] \in [t-D, t-1] \right\rbrace$. If $V^{t-1}_i[k] = 0$, $C^{t-1}_i[k] =0$, and $\mathcal{J}_k \neq \emptyset$, then set $V^{t}_i[k] = 1$ and $U^{t}_i[k] = \displaystyle \max_{j\in \mathcal{J}_k} \left\lbrace U^{t-1}_j[k] \right\rbrace$.
\end{modified_rules}
\begin{modified_rules}[R3']
Set $T^{t}_i = \displaystyle \max_{k\in \mathcal{A}} \left\lbrace U^{t}_i[k] \right\rbrace$.
\end{modified_rules}
At iteration $t \geq 2$, in addition to Rules R1'--R3', agent~$i$ applies the following rules:
\begin{modified_rules}[R4']
For each~$k\in \mathcal{A}$, if the termination vectors $V^{t-2}_i[k] = 1$ and $V^{t-1}_i[k] = 1$, and there is $V^{t-1}_j[k] = 0$ or $V^{t-1}_j[k] = 1$ with $U^{t-1}_j[k] \neq U^{t-1}_i[k]$ for any~$j\in \mathcal{N}(i)$, then set $V^{t}_i[k] = 0$, $U^{t}_i[k] = t$, $T^{t}_i = t$, and $C^t_i[k] = U^{t-1}_i[k] + 2D + |\mathcal{A}| - 1 - t$.
\end{modified_rules}%
\begin{modified_rules}[R5']
For each~$k\in \mathcal{A}$, if $C^{t-1}_i[k] > 0$, then set $C^{t}_i[k] = C^{t-1}_i[k] -1$.
\end{modified_rules}%

\begin{modified_rules}[R6']
If $\sum_{k\in \mathcal{A}}{V^{t}_i[k]} = |\mathcal{A}|$ and $t \geq T^{t}_i + 2D + |\mathcal{A}| - 1$, terminate. 
\end{modified_rules}

These rules are similar to the termination method in Section~\ref{sec3:methodology} with additional conditions that mandate clearing any faulty termination statuses. The main differences here are: 1)~only accepting updates that do not violate the conditions in Proposition~P5 as stated in Rule~R2', 2)~clearing any status that violates the conditions in Proposition~P4 as stated in Rule~R4', 3)~after a fault is detected, rejecting any new update for a certain number of iterations computed by Rule~R4' and updated by Rule~R5', and 4)~accounting for the passing of the correction termination status from the faulted agent to the other agents in the termination conditions as stated in Rule~R6'.

Selecting the number of iterations to reject new updates plays an important role in the correctness of the fault-tolerant method as we will show in the correctness proofs. When an agent detects a fault via Proposition~P4, the agent keeps the termination status vector equal to zero for the next $C^{t}_i[k] = U^{t-1}_i[k] + D + |\mathcal{A}| - 1 - t$ iterations due to Rules~R4' and R5'. During these iterations, agents clear the faulty termination status before allowing any change to the termination status. The use of the iteration vector $U^{t-1}_i[k]$ ensures that all agents use the same iteration number of the faulty status to compute $C^{t}_i[k]$. Thus, the value of $C^{t}_i[k]$ will be exactly the same for all agents that receive the same faulty status. 

The fault-tolerant method requires storing additional information from the previous iterations. Specifically, each agent needs to store the termination vectors from the last two iterations ($V_i\in \left\lbrace 0,1 \right\rbrace^{|\mathcal{A}|}$), the termination iteration vector ($U_i\in \mathbb{Z}_{\ge0}^{|\mathcal{A}|}$), the termination iteration scalar ($T_i\in \mathbb{Z}_{\ge0}$), and the correction vector ($C_i\in \mathbb{Z}_{\ge0}^{|\mathcal{A}|}$). Thus, the total stored data for agent~$i$ is $2|\mathcal{A}|$ binaries and $2|\mathcal{A}|+1$ non-negative integers. 

\subsection{Fault-Tolerant Method Correctness}
Since Proposition~P1 does not hold when there is a faulty termination status, we can not use Proposition~P3 to prove that the algorithm will terminate appropriately. We will instead use the following three propositions to show the correctness of the fault-tolerant extension. We first bound the maximum number of iterations for any faulty status to persist in the termination vectors. Then, we show that no faulty status can terminate the algorithms. Lastly, we show that the algorithm terminates appropriately if there are no new fault injections.

\begin{proposition}[P6]
The maximum number of iterations that any faulty termination status will persist in any termination vector ($V_i$, for any $i\in\mathcal{A}$) is at most $D + |\mathcal{A}| - 2$ iterations. 
\end{proposition}

We prove Proposition~P6 by tracing the path of the faulty termination status from the faulty agent to the faulted agent passing through an arbitrary agent using Rule~R2', and then tracing back the correction status from the faulted agent to the same arbitrary agent using Rule~R4'.

\begin{proof}
Without loss of generality, assume there is a faulty agent~$j$, a faulted agent~$k$ ($j \neq k$ due to Assumption~A5), and the fault is injected at iteration $t_j$. Consider any arbitrary agent~$i$ that receives the faulty status from agent~$j$ at iteration~$t_i$ and let $M = t_i - t_j$, i.e., the length of the shortest path between $j$ and $i$ such that the faulted agent $k$ is not in that path. We want to show that this faulty status will be cleared at iteration $t^{\prime}_i \leq t_i + D + |\mathcal{A}| - 2$. 

Let $P(i,k)= \left\lbrace i = i_0, i_1, \ldots , i_r, \ldots, i_R = k \right\rbrace$ be the shortest path between $i$ and $k$ with length $R$ such that $j \notin P(i,k)$ (this path exists due to Assumption A4). Agents along this path either receive or do not receive the faulty status at iteration $t_i = t_j + M$. At the next $S$ iterations, one or more agents in the path $P(i,k)$ will receive the faulty status for the first time at each iteration. At iteration $t_i + s$, where $s \leq S$, we can trace the agents that receive the faulty status along the path $P(i,k)$ starting from agent~$i$ until some agent $i_r \in P(i,k)$ that did not receive the faulty status yet. The next iteration $t_i + s + 1$, agent $i_r$ will either receive the faulty status or reject the faulty status. The rejection could happen because one of two possibilities: 1) $M + s = D + 1$ due to the condition in Rule R2' (reject any new update that takes more than $D$ iterations) or 2) agent~$i_r$ received and cleared the faulty status before iteration $t_i + s$ and $C_{i_r}^{t_i + s}[k] > 0$ (if $C_{i_r}^{t_i +s}[k] > 0$ agent $i_r$ rejects any new update due to Rule R2'). Eventually, there will be an agent $i_r\in P(i,k)$ that will reject the faulty status update at iteration $S$ with $M + S \leq D + 1$ since the faulted agent $k\in P(i,k)$. 

This implies the termination vector of agent $i_{r-1} \in P(i,k)$ (i.e., $i_{r-1}$ is the neighbor of $i_r$ in the path $P(i,k)$ with shortest distance to agent $i$) is $V_{i_{r-1}}^{t_i + S + 1}[k] = 0$ due to Rule R4'. The next iterations, agents along the path $P(i,i_{r})\subseteq P(i,k)$ will also clear the faulty status leading to $V_{i_{0}}^{t_i + S + r}[k] = 0$ due to Rule R4'. This implies agent $i$ will clear the faulty status at iteration $t^{\prime}_i = t_i + S + r \leq t_i + D + |\mathcal{A}| - 2$ since $S \leq D + 1 - M$, $M \geq 1$, and $r \leq R \leq |\mathcal{A}| - 2$. Thus, the maximum number of iterations for any faulty status to persist in agent~$i$'s termination vector is $t^{\prime}_i - t_i \leq D + |\mathcal{A}| - 2$.
\end{proof}

We can use the same proof to show that the same bound also applies when there are multiple faulty agents. With multiple faulty agents, we need to differentiate between two cases. If the faulty agents collude in the sense that they use the same value for the faulty termination iteration vector $U_j[k]$, then the exact same proof applies. On the other hand, if the faulty agents use different values for the faulty termination iteration vector $U_j[k]$, then when following the path $P(i,k)$ in the proof, an agent $i_{r-1}\in P(i,k)$ might clear the faulty status because of the discrepancy in the values of the termination iteration vectors with agent $i_r \in P(i,k)$ due to Rule R4'. Note that the bound in the proposition is the tightest bound for any general graph as shown in the example in Figure~\ref{fig:p6_example}. Next, we use Proposition~P6 to show that no faulty status can terminate the distributed algorithm.
\begin{figure}[htbp]
    \centering
    \includegraphics[width=3in]{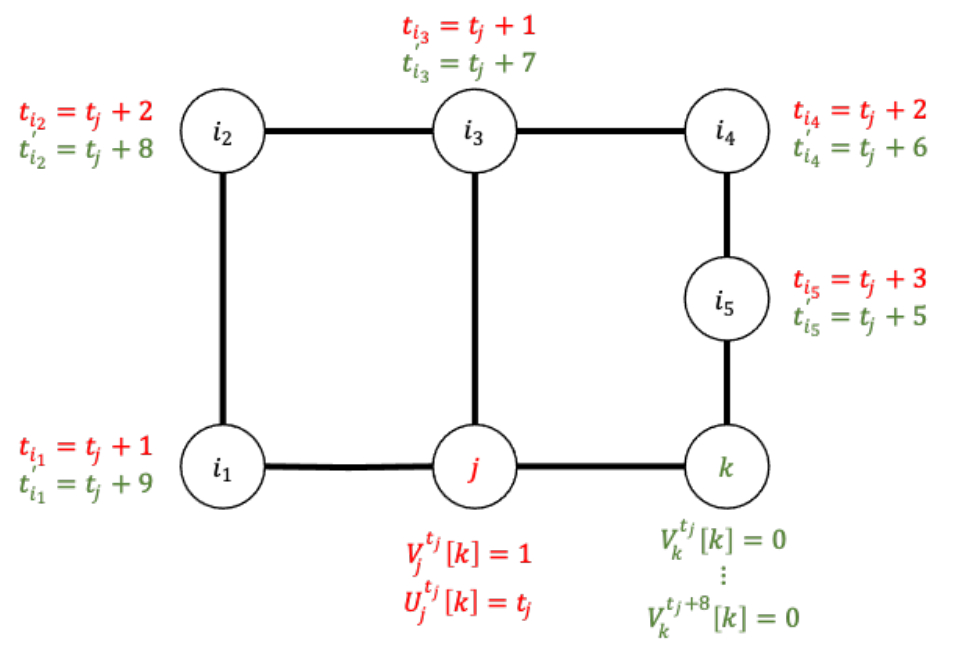}
	\caption[exampleP6]{An example for a faulty status where the bound in Proposition~P6 is tight. The number of agents $|\mathcal{A}| = 7$ and the communication network diameter~$D = 3$. The fault originates from agent~$j$ (in red) at iteration $t_j$ and the faulted agent is~$k$. At iteration~$t_j+3$, all agents receive the faulty status (in red). Agent~$i_5$ is the first agent that clears the faulty status since agent~$i_5$ is a neighbor to the faulted agent~$k$. Agents clear the faulty status at iteration $t^{\prime}_i$ (in green). At iteration~$t_j + 9$, agent~$i_1$ is the last agent to clear the faulty status. Thus, the number of iterations agent~$i_1$ needs to clear the faulty status is $D + |\mathcal{A}| - 2 = 8$ iterations.}
	\label{fig:p6_example}
\end{figure}

\begin{proposition}[P7]
No agent will terminate the computations before satisfying the global termination criterion.
\end{proposition}

The proof of Proposition~P7 is a direct consequence of Proposition~P6. We only need to show that any fault will be cleared before the termination condition in Rule~R6' is true. 

\begin{proof}
The only way for any agent to terminate the computation before satisfying the global termination criterion is to have a faulty status in the termination vector. Assume agent~$i$ receives a faulty termination status at iteration $t_i$. Since the maximum number of iterations for any faulty status to persist in the termination vectors is $D + |\mathcal{A}| - 2$ as stated in Proposition~P6, the iteration $t^{\prime}_i$ at which agent~$i$ clears the faulty status must be no later than iteration $t_i + D + |\mathcal{A}| - 2$. From Rules~R2' and R3', we know that $T_i^{t_i} \geq t_i - D$, otherwise agent $i$ will not receive the faulty status due to Rule~R2'. Thus, agent~$i$ terminates using Rule~R6' by at least iteration $t \geq T_i^{t_i} + 2D + |\mathcal{A}| - 1 \geq t_i + D + |\mathcal{A}| - 1$. However, since agent~$i$ clears the faulty status at iteration $t^{\prime}_i \leq t_i + D + |\mathcal{A}| - 2 < t_i + D + |\mathcal{A}| - 1 \leq t$, agent~$i$ will not terminate by any faulty termination status. Thus, the algorithm will not terminate before satisfying the global termination criterion.
\end{proof}

When a faulty agent keeps injecting faulty termination status, the algorithm will not terminate, as implied by Proposition~P7. However, if the faulty agent stops injecting faulty termination statuses, Rules~R4' and R5' ensure clearing the faulty termination status in all agents' local termination vectors ($V_i$, for all $i\in\mathcal{A}$). If all agents clear the faulty termination status, then Propositions~P1 and P2 are valid, and the algorithm terminates appropriately using Rule~R6'. The last proposition establishes that Rules~R4' and R5' clear the faulty termination status if there is no new fault injections.

\begin{proposition}[P8]
If the distributed algorithm satisfies the global termination criterion at iteration $T_G$ and there is no fault injections at any iteration $t \geq T_G$, then the algorithm will terminate appropriately.
\end{proposition}

We prove Proposition~P8 by showing that whenever there is a faulty status, there will be a future iteration where all agents clear that faulty status, and this iteration must be before any agent terminates the computation due to the faulty status. We achieve this result using the correction vector $C_i$, which prevents agents from accepting any new updates until waiting a certain number of iterations after detecting a faulty status.

\begin{proof}
Without loss of generality, assume there is a fault injection by a faulty agent~$j$ and a faulted agent~$k$ ($k\neq j$ due to Assumption A5). Further, assume the faulted agent is the last agent to satisfy its local termination criterion and the faulty agent stops injecting the faulty status before the faulted agent~$k$ satisfies its local termination criterion. Denote the $k$-th entry of the faulty termination iteration vector as $U_j[k] = t_j$. For any arbitrary agent~$i$ that receives the faulty status at iteration $t_i$, the local termination statuses are $V_i^{t_i}[k] = 1$, $U_i^{t_i}[k] = t_j$, and $t_i \leq t_j + D$ due to Rule R2'.

From Proposition~P6, agent $i$ clears the faulty status by iteration $t^{\prime}_i \leq t_i + D + |\mathcal{A}| - 2$. The local termination statuses at this iteration are $V_i^{t^{\prime}_i}[k] = 0$ and $C_i^{t^{\prime}_i}[k] = U_i^{t_i} + 2D + |\mathcal{A}| - 1 - t^{\prime}_i = t_j + 2D + |\mathcal{A}| - 1 - t^{\prime}_i$. Since $t^{\prime}_i \leq t_i + D + |\mathcal{A}| - 2$ and $t_i \leq t_j + D$, then $t^{\prime}_i \leq t_j + 2D + |\mathcal{A}| - 2$. This implies $C_i^{t^{\prime}_i}[k] \geq t_j + 2D + |\mathcal{A}| - 1 - (t_j + 2D + |\mathcal{A}| - 2) = 1 > 0$. Since $C_i^{t^{\prime}_i}[k] > 0$, agent~$i$ will not update its local termination status during the next $C_i^{t^{\prime}_i}[k]$ iterations. Let $t^{\prime \prime}$ be the last iteration where agent $i$ will not update its termination status. We have $t^{\prime \prime} = t^{\prime}_i + C_i^{t^{\prime}_i}[k] = t_j + 2D + |\mathcal{A}| - 1$. Since $t^{\prime \prime}$ only depends on the faulty status, then the termination vector of any agent $i\in\mathcal{A}\setminus \left\lbrace j \right\rbrace$ that receives the faulty status is $V_i^{t^{\prime \prime}}[k] = 0$. Thus, all agents clear the faulty status and do not accept any updates until iteration $t^{\prime \prime}$ if the faulted agent $k$ does not satisfy its termination criterion.

Let $t_k$ be the iteration when the faulted agent $k$ satisfies its local termination criterion. If $t_k \geq t^{\prime \prime}$, then we can use Propositions P1--P3 to show that the algorithm terminates appropriately. We still need to show that the algorithm terminates appropriately if $t_k < t^{\prime \prime}$. 

If $t_k < t^{\prime \prime}$, we can partition the agents into two sets based on whether they received the correct termination status. Let $\mathcal{I} = \left\lbrace i \mid i\in \mathcal{A}, V^{t_i}_i[k] = 1,  U^{t_i}_i[k] = t_k, t_i \leq t_k + D  \right\rbrace$ be the set of agents that receive the correct termination status before iteration $t_k + D$, and let $\mathcal{I}^\mathsf{c} = \mathcal{A} \setminus \mathcal{I}$ be the remaining agents. 

If $\mathcal{I}^\mathsf{c} = \emptyset$, then we can use Propositions~P1--P3 to show that the algorithm terminates appropriately. On the other hand, if $\mathcal{I}^\mathsf{c} \neq \emptyset$, then we know that there is an iteration $t^{\prime \prime}$ where $V_i^{t^{\prime \prime}}[k] =0$, for all $i \in \mathcal{I}^\mathsf{c}$. Since $\mathcal{I}$ and $\mathcal{I}^\mathsf{c}$ form a graph partition, we always have some agent $i\in \mathcal{N}(i')$ where $i\in\mathcal{I}$ and $i'\in\mathcal{I}^\mathsf{c}$. Thus, agents in $\mathcal{I}$ will also clear the correct termination status due to Rule R4'. Moreover, those agents will not receive any new updates until some iteration $t^{\prime \prime}_k = t_k + 2D + |\mathcal{A}| -1$. Since $t_k > t_j$ by assumption, then $t^{\prime \prime}_k > t^{\prime \prime}$. Thus, we have $V^{t^{\prime \prime}_k}_i[k] =0$, for all $i \in \mathcal{A}$. We can then use Propositions~P1--P3 to show that the algorithm terminates appropriately in this case.
\end{proof}

The same proof applies when there are multiple faulty agents. The only restrictions are that the faulty agents do not form a cut-set as Assumption A4 states, and the faulty agents do not terminate the computation before satisfying their local termination criterion as Assumption A5 implies.

\section{Simulation Results} \label{sec5:result}
To demonstrate the proposed termination method, we use the Alternating Direction Method of Multipliers (ADMM) algorithm to solve the Optimal Power Flow problem (OPF) problem. OPF problems are used to optimize the dispatch of generators in electric power grids. We consider a 240-bus test case with 22 areas from PGLib-OPF benchmark library~\cite{PGLib}. Each agent controls an area and communicates with neighboring agents when there is an electric transmission line connecting the two areas. As shown in Figure~\ref{fig:network}, the communication network consists of 36 links and has a diameter $D=7$. We decompose the power system into multiple subsystems using the method in~\cite{pmada2023} and, for illustrative purposes, solve the DC approximation of the OPF problem. For the local termination criterion, we use the shared variables' mismatches (i.e., the differences in the voltage phase angles of the shared buses between neighboring agents) having an $\ell_{\infty}$-norm less than $10^{-2}$ radians.
\begin{figure}
    \centering
    \includegraphics[width=\linewidth]{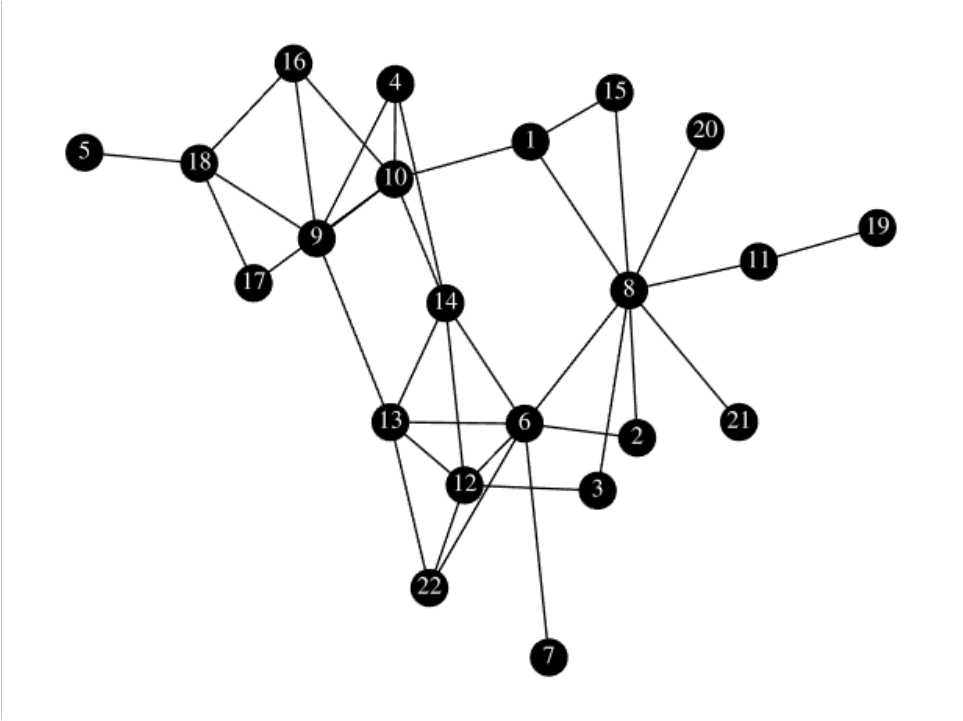}
    \caption{Communication network of the 240-bus power system with 22 agents. The nodes are the agents and the edges are the communication links.}
    \label{fig:network}
\end{figure}

The agents satisfy the global termination criterion for the first time after 862 iterations. Agent~8 is the last agent to satisfy its local termination criterion at this iteration. The algorithm then terminates using the proposed method at iteration 869, seven iterations after the the global termination criterion is satisfied. Figure~\ref{fig:termination} shows the agents' termination statuses for the last nine iterations. Information regarding satisfaction of the global termination status takes five iterations to traverse all the agents. However, based on the diameter of the communication network ($D=7$), the agents perform two additional iterations before terminating to ensure that all the agents are informed about satisfaction of the global termination criterion. 

\begin{figure*}
    \centering
    \includegraphics[width=\linewidth]{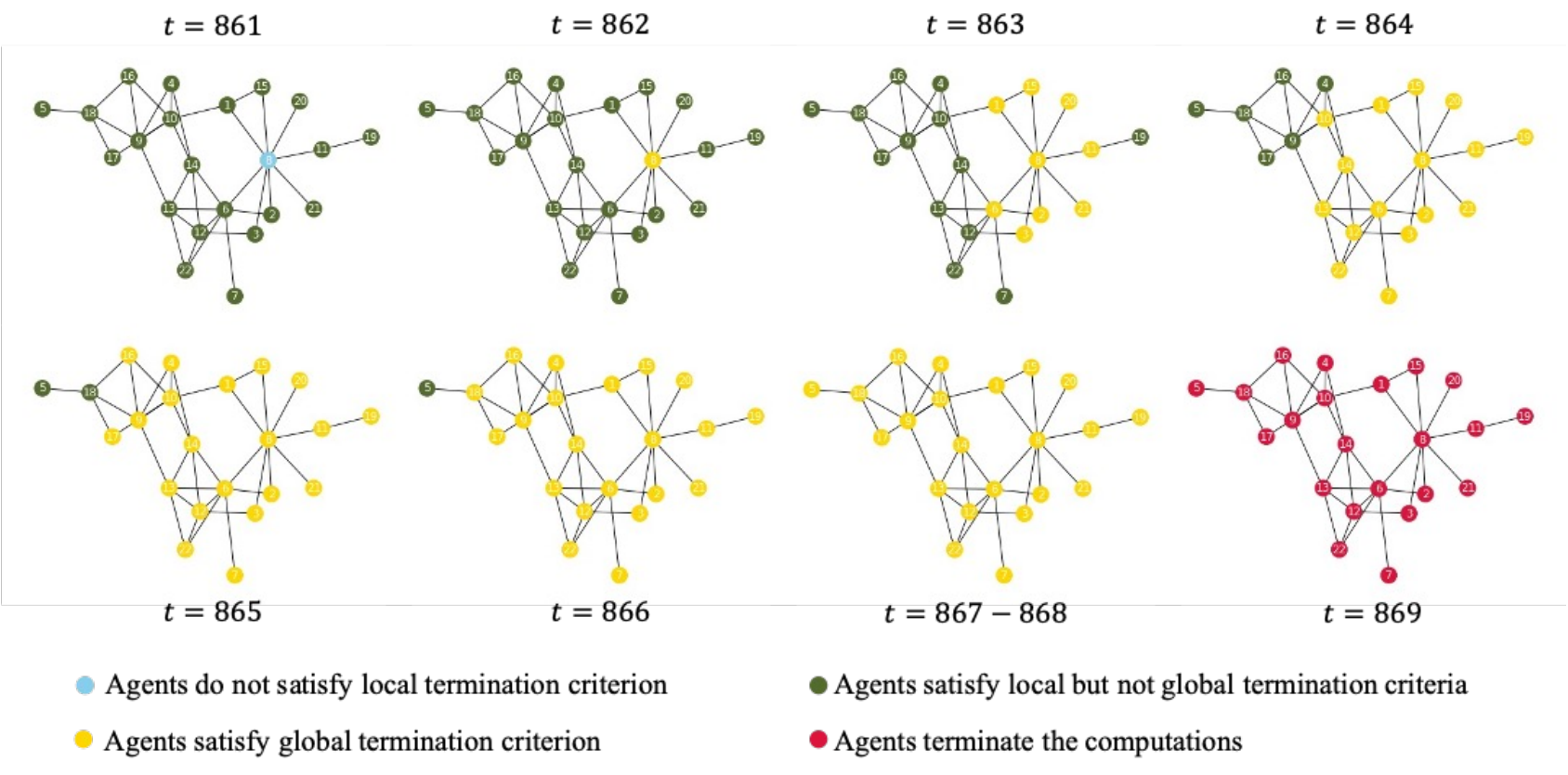}
	\caption[attacker]{Termination status for the last nine iterations of the 240-bus optimal power flow test case solved using the alternating direction method of multipliers and the proposed termination method. At iteration $t=861$, all agents satisfy their local termination criterion (green nodes) expect agent~8 (cyan node). Agent~8 satisfies its local termination criterion at iteration $t=862$ and acknowledges the global termination criterion (becomes yellow). During iterations $t=863~\text{--}~868$, the global termination criterion then traverses the other agents (they become yellow). In iterations~$t=867~\text{--}~868$, all agents have received the global termination criterion (all nodes are yellow) but the agents do not terminate the computation since they are not yet assured that the global termination status information has reached all agents. This occurs at iteration~$t=869$, at which point all agents terminate the computation appropriately (red nodes).}
	\label{fig:termination}
\end{figure*}

To demonstrate the fault-tolerant extension of the termination method described in Section~\ref{sec4:fault_termination}, we next consider a case where some agents inject faulty termination statuses. Faulty agents inject a faulty termination every $100$ iterations by setting $V_f^{t_f}[k] =1$, for all $k \in \mathcal{A}$, $f \in \mathcal{F}$, and $t_f \in \left\lbrace 100, 200, 300, 400, 500, 600, 700, 800 \right\rbrace$, where $\mathcal{F}$ is the set of faulty agents. Moreover, we have the faulty agents send faulty status for 20 iterations. To maximize the possibility of terminating the algorithm, we set the value of the termination iteration vector to be $U_f^{t_f}[k] = U_f^{t_f-1}[k]$, for all $k \in \mathcal{A}$, if $V_f^{t_f-1}[k] =1$ and $U_f^{t_f}[k] = t_f$ otherwise. We stop injecting the faulty termination status after iteration 820 to verify that the algorithm will then terminate appropriately. 

Table~\ref{tab:faulty_result} summarizes the simulations for cases with one to five faulty agents. The distributed algorithm terminates in 897 iterations in all cases. The maximum number of iterations that the agents presume the global termination criterion is satisfied because of the faulty termination status is seven iterations. Furthermore, the maximum number of iterations for any single fault to persist in any local termination vector other than the faulty agent is eight iterations. Despite the presence of faulty agents, the algorithm terminates appropriately in all cases.

Figure~\ref{fig:faulty} shows the number of agents that satisfy the local and the global termination criteria when five faulty agents collude in an attempt to terminate the computation before the global termination criterion is satisfied. Observe that the number of agents that receive the faulty global termination criterion spikes after any faulty status is injected by the faulty agents and returns to zero after a few iterations.
\begin{table}[h]
\centering
\caption{Faulty Test Case Results Summary}
\begin{tabular}{|c|c|c|c|}
\hline
Faulty Agents &
  \begin{tabular}[c]{@{}c@{}}Termination \\ Iteration\end{tabular} &
  \begin{tabular}[c]{@{}c@{}}Maximum Single \\ Fault Persists\end{tabular} &
  \begin{tabular}[c]{@{}c@{}}Maximum Global\\ Fault Persists\end{tabular} \\ \hline\hline
1              & 897 & 7 & 6 \\ \hline
1, 3           & 897 & 7 & 5 \\ \hline
1, 3, 4        & 897 & 7 & 5 \\ \hline
1, 3, 4, 9     & 897 & 8 & 7 \\ \hline
1, 3, 4, 9, 14 & 897 & 8 & 7 \\ \hline
\end{tabular}
\label{tab:faulty_result}
\end{table}
\begin{figure}[h]
    \centering
    \includegraphics[width=\linewidth]{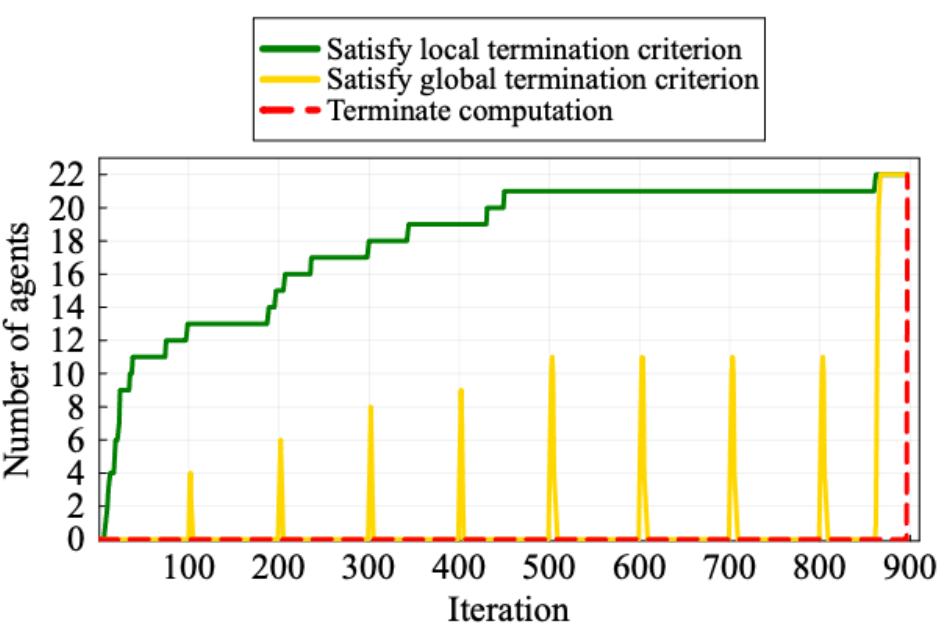}
	\caption[faulty]{Number of agents that satisfy the local (green) and global (yellow) termination criteria and terminate computation (red) when five faulty agents inject faulty statuses for the first 20 of every 100 iterations, stopping at iteration 820. The algorithm then appropriately terminates at iteration 897.}
	\label{fig:faulty}
\end{figure}

\section{Remarks and Extensions} \label{sec6:discussion}
This section highlights several remarks and introduces extensions of the proposed method to reduce the computational burden, detect faulty agents, revise assumptions, and terminate asynchronous distributed optimization algorithms.

\subsection{Reducing Additional Iterations}
The proposed method in Section~\ref{sec3:methodology} terminates distributed optimization algorithms after $D$ iterations of satisfying the global termination criterion to ensure that the correct termination status is communicated to all agents. Moreover, the fault-tolerant termination method in Section~\ref{sec4:fault_termination} uses an additional $2D + |\mathcal{A}|-1$ iterations to prevent faulty early termination. The additional $D+|\mathcal{A}|-1$ iterations are thus a trade-off between reliability and computational burden. 

We can reduce the number of computational iterations per agent while maintaining the fault tolerance guarantees. Rather than performing the local computations after an agent satisfies the global termination criterion, the agent can stop performing the normal computations and instead only communicate the termination status. In this case, once all agents are assured to have received the global termination status, the agents terminate and use the solution from iteration $T_G$ instead of the last iteration's solution. If there are any discrepancies in the termination statuses due to a faulty agent, the agents continue from where they stopped. 

This modification will reduce the number of additional computations needed in the fault-tolerant method by more than $66\%$. This reduction comes from the fact that all agents receive the global termination status in a maximum of $D$ iterations. In the remaining $D + |\mathcal{A}|-1$ iterations, the agents participate in the termination method only, which has a negligible computational burden compared with solving the agent's local optimization problem. Thus, the fraction of avoided additional computations is $\frac{D+ |\mathcal{A}|-1}{2D + |\mathcal{A}|-1}$, which is lower bounded by 0.66.

\subsection{Detecting Faulty Agents}
The fault-tolerant method self-corrects any faulty status but does not provide information about the faulty agents. If faulty agents continue sending a faulty termination status, we can show that the faulty agents can be detected by their neighbors. Specifically, for some iteration~$t$ and agent~$i$, if $V^{t-2}_i[k] = 1$, $V^{t-1}_i[k] = 0$, and $V^{t}_j[k]=1$ for some neighbor $j\in N(i)$, then $j$ is a faulty agent since it does not follow Rule~R4'. In other words, if an agent clears a fault at a certain iteration, and there are one or more neighbors that do not clear the same fault in the subsequent iteration, then those neighbors are faulty agents.

\subsection{Revising Assumptions}\label{sec:revising_assumptions}
To prove the termination method's correctness, we use Assumption~A3 which states that the termination criterion is monotone. However, we can relax this assumption in the fault-tolerant method. We can treat any non-monotone behavior of the local termination criterion status as a fault injection. The fault-tolerant method will then self-correct the status to reflect the change in the non-monotone local termination criterion status. The only required assumption for the fault-tolerant method with a non-monotone local termination criterion is that if the local termination criterion is satisfied for $D +|\mathcal{A}| - 1$ iterations then it remains satisfied for all subsequent iterations, which is weaker than Assumption~A3.

As stated in Assumption~A4, the fault-tolerant method requires having a path between any two agents that does not pass through a faulty agent. This suggests that the best way to prevent faulty termination when using the proposed method is to increase the minimum number of agents in any cut-set by increasing the communication network's connectivity. If the communication network is $k$-connected, i.e., there are $k$ disjoint paths between any pairs of agents, then we are assured that the distributed algorithm will avoid faulty termination with up to $k-1$ faulty agents.

Assumption~A5 isolates faults occurring in the termination method from faults in the distributed optimization algorithm itself. Although we use Assumption~A5 to prove the correctness of the fault-tolerant method, we can relax this assumption. The proofs of Propositions~P6--P8 assume the faulty and faulted agents are not the same agent. However, the only assumption we need in these proofs is that the faulty agent is not the last agent that satisfies the local termination criterion.

\subsection{Asynchronous Distributed Optimization Algorithms}
The proposed termination method is presented in the framework of synchronized distributed optimization where all agents communicate at common intervals. However, the method can be modified to terminate asynchronous distributed optimization algorithms. The modification requires a synchronized time measurement that records the time at which the last local termination is satisfied. This time measurement replaces the termination iteration scalar ($T_i$, for all $i\in\mathcal{A}$). Furthermore, the agents need to store the local solutions with a time stamp for each solution at each local iteration. 

Asynchronous distributed algorithms can be appropriately terminated as follows. When an agent satisfies the local termination criterion, the agent records the time in $T_i$ and shares the termination status ($T_i$ and $V_i$) with the neighboring agents. Agents then use Rule~R2 to update the termination vector ($V_i$) and update the termination time ($T_i$) similar to the termination iteration scalar. Lastly, agents use the first condition in Rule~R3 to terminate, that is, terminate if $\sum_{k\in \mathcal{A}}{V^{t}_i[k]} = |\mathcal{A}|$. The solution of the distributed algorithm in this case is the last local solution with a timestamp equal to or before the time in the local termination status ($T_i$). We note that this extension works for asynchronous distributed optimization termination without faulty agents, but further work is needed to extend the fault-tolerant method to asynchronous settings.

\section{Conclusion} \label{sec7:conclusion}
This paper proposes a method for terminating distributed optimization algorithms solved by multiple computing agents. The proposed method uses three simple rules for each agent that solely rely on local information. Furthermore, we propose a fault-tolerant extension to the method via three additional rules. These rules are symmetrical (same for every agent), and no central entity is required to terminate the distributed optimization algorithm. Further, the proposed method requires minimal knowledge of the global system. The only prior information that we assume the agents know are the number of agents and the diameter of the communication network.

There are several directions for future work that build on the method proposed in this paper. First, we note that proving the method's correctness requires the local termination criterion to be monotone, i.e., once satisfied, each agent's local termination status remains satisfied for all future iterations. However, the local termination criterion may not be monotone for some distributed optimization algorithms. We therefore plan to analyze local termination criterion satisfaction for various distributed optimization algorithms and design alternate criteria that exhibit monotone behavior. Second, our contemporaneous work in~\cite{alkhraijah2023detecting} studies faults in the distributed optimization algorithm's computations. By combining ideas in~\cite{alkhraijah2023detecting} with the results of this paper, we aim to simultaneously address faults in both the distributed optimization algorithm's computations and the termination method. Third, we plan to extend the fault-tolerant method to consider asynchronous distributed optimization as well as dynamic communication networks where connections between agents may vary across iterations.

\section*{Acknowledgment}

The authors would like to thank Rachel Harris for insightful discussions and feedback on the paper.

\bibliographystyle{IEEEtran}
\bibliography{IEEEabrv, references.bib}

\end{document}